\documentclass[12pt, a4paper]{amsart}
 \usepackage{amsaddr}
 \usepackage{tikz}
 \usetikzlibrary{decorations.markings}
\usepackage[onehalfspacing]{setspace}
 \usepackage{stmaryrd}
\usepackage{fullpage}
\usepackage{amsmath}
\usepackage{epigraph}
\usepackage{amsmath, amscd}
\usepackage{IEEEtrantools}
\usepackage{amssymb, tikz}
\usepackage{mathabx}
\usepackage{mathtools}
\usepackage{comment}

\calclayout

\newtheorem{thm}{Theorem}[section]
\newtheorem{prop}[thm]{Proposition}
\newtheorem{lem}[thm]{Lemma}
\newtheorem{cor}[thm]{Corollary}

\theoremstyle{definition}

\theoremstyle{definition}

\theoremstyle{remark}

\newtheorem{remark}[thm]{Remark}

\numberwithin{equation}{section}

\def\G{\Gamma}

\newcommand{\R}{\mathbf{R}}  % The real numbers.
\newcommand{\Z}{\mathbf{Z}}
\newcommand{\Q}{\mathbf{Q}}

\newcommand{\Sf}{\mathfrak{S}}

\newcommand{\SL}{\mathrm{SL}_2 (\mathbf{Z})}

\newcommand{\llrrparen}[1]{% \llrrparen{..}
  \left(\mkern-6mu\left(#1\right)\mkern-6mu\right)}

\newcommand{\floor}[1]{\left \lfloor #1 \right \rfloor}

\newcommand{\uh}{\mathbf{H}}

\newcommand{\sgn}{\mathrm{sign}}

\newcommand\SmallMatrix[4]{{\tiny\arraycolsep=0.3\arraycolsep\ensuremath{\begin{pmatrix}#1 & #2 \\ #3 & #4\end{pmatrix}}}}

\def\XXint#1#2#3{{\setbox0=\hbox{$#1{#2#3}{\displaystyle\displaystyle\int}$ }
\vcenter{\hbox{$#2#3$ }}\kern-.6\wd0}}

\usepackage{blindtext}
\usepackage{soul}

\usepackage{etoolbox}

\makeatletter
\patchcmd{\@setauthors}{\MakeUppercase\@author}{}{}{}
\makeatother

\begin{document}

\title{Two Geometric Interpretations of Hardy Sums}
\author{Alessandro L\"ageler}
\maketitle

\begin{abstract}
    The problem of finding the number of lattice points in a triangle has a classical solution if the lattice is $\Z^2$ and the vertices of the triangle have integer valued coordinates. We consider what happens when we replace the lattice by $(2 \Z)^2$ instead and give an explicit formula for the number of lattice points inside a triangle in terms of Hardy sums. Moreover, we give a second geometric interpretation of the Hardy sums as signed intersection numbers with a certain oriented net of geodesics. Using this geometric realization, we prove a generalized reciprocity law for Hardy sums by an elementary argument. 
\end{abstract}

\section{Introduction}

Problems of counting points with integer coordinates in geometric objects are challenging and well-studied problems in mathematics and give an interesting connection between number theory and geometry (famous examples include counting integer points on the $3$-sphere, the Gauss circle problem or Hardy and Littlewood's work on counting points in a right-angled triangle \cite{hardy}). These problems are typically very hard. However, when one considers two-dimensional polygons whose vertices have integer coordinates, the problem has an elegant and simple solution. 

\begin{thm}[Pick \cite{pick}] \label{pick}
Let $P$ be a polygon whose vertices lie in $\Z^2$. We have $$\mathrm{area}(P) = \#( P^\circ \cap \Z^2) + \frac{1}{2} \# (\partial P \cap \Z^2) + 1,$$
where $P^\circ$ denotes the interior and $\partial P$ denotes the boundary of $P$. 
\end{thm}

In this note, we limit ourselves to the case of triangles. For a pair of positive integers $d, c$, let $T(d, c)$ be the triangle in $\R^2$ with vertices $(0, 0)$, $(d, 0)$, and $(0, c)$. The set $T(d, c)$ is the set of points $(x, y) \in \R^2$ such that $0 \leq \frac{x}{d} + \frac{y}{c} < 1$ (we exclude the acute angles and the hypotenuse). For simplicity, we shall assume that $(d, c) = 1$.

By Theorem \ref{pick}, we only need to identify the points on the boundary of $T(d, c)$ if we want to find $\# (T(d, c) \cap \Z^2)$, namely \begin{equation} \label{pickcount}
    \# (T(d, c) \cap \Z^2) = \mathrm{area}(T(d, c)) + \frac{1}{2} \# ( \partial T(d, c) \cap \Z^2 ) - 1.
\end{equation}
As we assumed that $d, c$ are coprime there are no lattice points $(x, y) \in \Z^2$ that satisfy $\frac{x}{d} + \frac{y}{c} = 1$ with both $x, y$ non-zero. Hence, the points on the boundary are given by $(0, 0), (1, 0), (0, 1), ...$, i.e. $\# ( \partial T(d, c) \cap \Z^2 ) = c + d + 1$.

We then find by equation (\ref{pickcount}): \begin{equation} \label{firstlatprob}
    \# ( T(d, c) \cap \Z^2 ) = \frac{1}{2} (c + 1) (d + 1) - 1.
\end{equation}

The problem of finding the lattice points in $T(d, c)$ becomes more difficult if we consider the lattice $(2\Z)^2$ instead of $\Z^2$. Obviously, the problem can be rescaled and solved by Pick's Theorem if both $d, c$ are even. So again, let us assume that $d, c$ are positive coprime integers. In this case, at least one of the vertices of $T(d, c)$ is no longer a lattice point and we cannot use Theorem \ref{pick}.

Here, we shall prove that $\# ( T(d, c) \cap (2\Z^2) )$ can be expressed in terms of the Hardy sums \begin{equation} \label{s4defi}
    S_4(d, c) = \sum_{k = 1}^{c - 1} (-1)^{\floor{\frac{kd}{c}}} \; \; \mathrm{for} \; (d, c) = 1, \; c > 0, \; d \; \mathrm{odd}.
\end{equation}

The Hardy sums $S_4(d, c)$ are integer-valued analogs of the classical Dedekind sums \begin{equation}
    s(d, c) = \sum_{k = 1}^{c - 1} \llrrparen{\frac{k}{c}} \llrrparen{\frac{kd}{c}} \; \; \mathrm{for} \; (d, c) = 1, \; c > 0,
\end{equation}
where $((x)) = x - \floor{x} - \frac{1}{2}$ for $x \notin \Z$ and $((x)) = 0$ for $x \in \Z$. Equivalently, the Dedekind sums $s(d, c)$ are uniquely determined by the following properties: \begin{enumerate}
    \item $s(0, 1) = 0$, 
    \item $s(d + c, c) = s(d, c)$, 
    \item $s(d, c) + s(c, d) = \frac{1}{12} \frac{c^2 + d^2 + 1}{cd} - \frac{1}{4}$ (the reciprocity formula).
\end{enumerate}

One can, in fact, express the Hardy sums $S_4(d, c)$ as a non-trivial linear combination of Dedekind sums, namely \begin{equation} \label{exprhardydede}
    S_4(d, c) = 8 \; s(d, 2c) - 4 \; s(d, c) \; \; \mathrm{for} \; (d, c) = 1, \; d \; \mathrm{odd}. 
\end{equation}
Equation (\ref{exprhardydede}) can be proved in an elementary way \cite{sitara} and is well-understood in the context of Eisenstein cohomology \cite[Ch. 2]{stevens}. 

The expression of $\# ( T(d, c) \cap (2 \Z)^2 )$ for $(d, c) = 1$ in terms of $S_4(d, c)$ is given in the next theorem. 

\begin{thm} \label{mainthm}
Let $d, c$ be positive coprime integers. \begin{enumerate}
    \item If $d, c$ are both odd, then $$\# ( T(d, c) \cap (2\Z)^2) = \frac{1}{8} (cd + 2d + 2c + S_4(d, c) + S_4(c, d)).$$
    \item If $d$ is even, then $$\# ( T(d, c) \cap (2\Z)^2)  = \frac{1}{4}\left(c + d + \frac{cd}{2} - \frac{1}{2cd} + S_4(c, d / 2) \right).$$
\end{enumerate}
\end{thm}

The proof of Theorem \ref{mainthm} relies on a theorem of Mordell \cite{mordell}, which we will present in the next section.

\medskip

The properties of Dedekind sums can be explained from a cohomological viewpoint as well. Consider the integer-valued Dedekind symbol \begin{equation} \label{phidefi}
    \Phi\left( \SmallMatrix{a}{b}{c}{d} \right) = \begin{cases} b, &\mathrm{for} \; c = 0, \\ \frac{a + d}{c} - 12 \; s(d, c), &\mathrm{for} \; c \neq 0. \end{cases}
\end{equation}
Asai \cite{asai} constructed the Dedekind symbol as a splitting of the cohomology class in $H^2(\mathrm{SL}_2(\R), \R)$ associated with the central extension of the universal covering group. The space $H^2(\mathrm{SL}_2(\R), \R)$ is a one-dimensional $\R$-vector space and generated by the cocycle \begin{equation} \label{radecocycledefi}
    w(A, B) = - \sgn(c_A c_B c_{AB}) 
\end{equation}
for $A = \SmallMatrix{*}{*}{c_A}{*}, \; B = \SmallMatrix{*}{*}{c_B}{*} \in \SL$,  and $AB = \SmallMatrix{*}{*}{c_{AB}}{*}$. Following the terminology of Kirby and Melvin \cite{kirbymelvin}, we call $w(A, B)$ the \textit{Rademacher cocycle}. Asai's "splitting" refers to the fact that there exists a unique function $\Phi : \SL \to \R$ such that $$w(A, B) = \Phi(AB) - \Phi(A) - \Phi (B),$$
which is exactly the Dedekind symbol $\Phi$ in (\ref{phidefi}). The uniqueness of $\Phi$ follows easily from $H^1(\SL, \R) = \{ 0 \}$. The reciprocity law for Dedekind sums corresponds to the case where $B = S = \SmallMatrix{0}{-1}{1}{0}$.

The Hardy sums $S_4(d, c)$, defined in (\ref{s4defi}), and \begin{equation} \label{sdefi}
    S(d, c) = \sum_{k = 1}^{\vert c \vert - 1} (-1)^{\floor{\frac{kd}{c}} + k + 1}, \; \; (d, c) = 1, \; c + d \; \mathrm{odd},
\end{equation}
however, need to be viewed in terms of the cohomology of certain subgroups of the modular group, namely \begin{align*}
    &\G^0(2) = \left\{ \SmallMatrix{a}{b}{c}{d} \in \SL : b \equiv 0 \pmod 2 \right\} \; \mathrm{resp.} \\
    &\G_\theta = \left\{ \SmallMatrix{a}{b}{c}{d} \in \SL : a \equiv d \pmod 2, \; b \equiv c \pmod 2 \right\}.
\end{align*}
Clearly, we can view the Hardy sums as functions on the subgroups $\G_\theta$ and $\G^0(2)$ directly by defining $S(A) = S(d, c)$ for $A = \SmallMatrix{*}{*}{c}{d} \in \G_\theta$ with $c \neq 0$. For the subgroups $\G_\theta$ and $\G^0(2)$, the first cohomology group is non-trivial.

Hardy sums appear as the correction factors in the transformation law of the logarithms of $\theta$-functions, so can be studied from an automorphic point of view. In section \ref{geomconstsec}, however, we will show that these Hardy sums can be defined as a signed intersection number, without any reference to automorphic forms; see Theorem \ref{geomchar}. We use this geometric interpretation to prove by an elementary argument that, say, $S(A)$ on $\G_\theta$, satisfies $$S(AB) - S(A) - S(B) = w(A, B)$$ for $A, B \in \G_\theta$, where $w(A, B)$ is the Rademacher cocycle (\ref{radecocycledefi}). Equivalently, this identifies the Hardy sums as restrictions of the Dedekind symbol $\Phi$ to the subgroups $\G_\theta$ resp. $\G^0(2)$ up to addition by a homomorphism.

\begin{thm} \label{mainthm1}
The functions $\Phi - S : \G_\theta \to \Z$ and $\Phi - S_4 : \G^0(2) \to \Z$ are homomorphisms. 
\end{thm}

We would like to remark that in an upcoming paper of Burrin and von Essen \cite{burvoness}, Burrin's generalized Dedekind sums \cite{bur1, bur2} are similarly realized as winding numbers of closed geodesics on the modular surface around the cusp $i \infty$.

\section{Proof of Theorem \ref{mainthm}}

To give a closed expression of $\# ( T(d, c) \cap (2\Z)^2)$, we consider the problem as a three-dimensional one. Namely, we reformulate \begin{align} \label{reformulateprob}
    \begin{split}
    \# ( T(d, c) \cap (2\Z)^2) &=  \# \left\{ (x, y) \in \Z^2 : 0 \leq \frac{2x}{d} + \frac{2y}{c} < 1 \right\} \\
    &= \# \left\{ (x, y) \in \Z^2 : 0 \leq \frac{x}{d} + \frac{y}{c} + \frac{1}{2} < 1 \right\}.
    \end{split}
\end{align}

For positive integers $u, v, w$, let $D(u, v, w)$ be the tetrahedron in $\R^3$ with vertices $(u, 0, 0)$, $(0, v, 0)$, and $(0, 0, w)$, i.e. $$D(u, v, w) = \left\{ (x, y, z) \in \R^3 : 0 < \frac{x}{u} + \frac{y}{v} + \frac{z}{w} < 1 \right\}.$$ 
Note that we excluded the vertices in our definition of $D(u, v, w)$.

Counting $\# ( T(d, c) \cap (2\Z)^2)$ is equivalent to counting the lattice points in $\Z^3$ of the tetrahedron $D(d, c, 2)$ with vertices $(d, 0, 0)$, $(0, c, 0)$, and $(0, 0, 2)$ intersected with the plane $\{(x, y, z) \in \R^3 : z = 1 \}$ by (\ref{reformulateprob}). Geometrically, this corresponds to projecting $T(d, c) \cap (2 \Z)^2$ onto the plane $\{ (x, y, z) \in \R^3 : z = 1 \}$ with base point $(0, 0, 2)$.

Mordell expressed $\# (D(u, v, w) \cap \Z^3)$ in a pretty closed form in terms of Dedekind sums, which we will use for our solution to the problem. 

\begin{thm}[Mordell] \label{mordell}
Let $u, v, w$ be pairwise coprime positive integers. We have \begin{align*}
    \# (D(u, v, w) \cap \Z^3) &= \frac{uvw}{6} + \frac{1}{4} (uv + uw + vw) + \frac{1}{4} (u + v + w) + \frac{1}{12} \left( \frac{uv}{w} + \frac{uw}{v} + \frac{vw}{u} \right) \\
    &+ \frac{1}{12 uvw} - 2 - (s(uv, w) + s(uw, v) + s(vw, u)).
\end{align*}
\end{thm}

\begin{remark}
One should note that the right hand side in Theorem \ref{mordell} is computationally easier than the left hand side. Indeed, if one were to calculate the left hand side without exploiting symmetries, the computational complexity would be $\mathcal{O}(uvw)$. The right hand side, however, is much faster to compute, as Dedekind sums can be computed through the Euclidean algorithm (see the reciprocity formula in the introduction).
\end{remark}

\subsection{Proof of the First Part of Theorem \ref{mainthm}} \label{firstpart}

Let $d, c$ be odd and coprime integers. By Theorem \ref{mordell} we have \begin{align} \label{n3first}
    \begin{split}
    \# (D(d, c, 2) \cap \Z^3) &= \frac{cd}{3} + \frac{1}{4} (2c + 2d + cd) + \frac{1}{4} (c + d + 2) + \frac{1}{12} \left( \frac{2c}{d} + \frac{2d}{c} + \frac{cd}{2} \right) \\
    &+ \frac{1}{24cd} - 2 - (s(2d, c) + s(2c, d) + s(cd, 2)).
    \end{split}
\end{align}
Since $cd$ is odd, we have $s(cd, 2) = s(1, 2) = 0$. 

We split the number of lattice points in $D(d, c, 2)$ into those contained in the plane $\{(x, y, z) \in \R^3  : z = 0 \}$ and those contained in the plane $\{(x, y, z) \in \R^3  : z = 1 \}$. The latter is, as we already said, equal to $\# ( T(d, c) \cap (2\Z)^2 )$. The former is equal to $$\# \left\{ (x, y) \in \Z^2 : 0 < \frac{x}{d} + \frac{y}{c} < 1 \right\} = \frac{1}{2} (c + 1) (d + 1) - 2$$
by (\ref{firstlatprob}) with the exclusion of the point $(0, 0)$. 

To write $\# ( T(d, c) \cap (2\Z)^2)$ in terms of Hardy sums, we use the following easy lemma.

\begin{lem} \label{hardybothodd}
Let $d, c$ be odd and coprime integers. We have \begin{equation*} \label{hardysumsdede}
    s(2c, d)  + s(2d, c) = \frac{4c^2 + 4d^2 + 1}{24cd}  - \frac{1}{8} (S_4(d, c) + S_4(c, d)) - \frac{3}{8}.
\end{equation*}
\end{lem}

\begin{proof}
We use formula (\ref{exprhardydede}) for the Hardy sum $S_4(d, c)$. A calculation yields \begin{align*}
    S_4(d, c) + S_4(c, d) &= 8s(d, 2c) - 4 s(d, c) + 8s(c, 2d) - 4 s(c, d) \\
    &= 8 (s(d, 2c) + s(c, 2d)) - 4(s(d, c) + s(c, d))\\
    &= 8\left(- s(2c, d) - s(2d, c) + \frac{1}{12} \frac{4c^2 + d^2 + 1}{2cd} + \frac{1}{12} \frac{c^2 + 4d^2 + 1}{2cd} - \frac{1}{2} \right) \\
    &- 4 \left( \frac{1}{12} \frac{c^2 + d^2 + 1}{cd} - \frac{1}{4} \right) \\
    &= \frac{4c^2 + 4d^2 + 1}{3cd} - 3 - 8\left(s(2c, d)  + s(2d, c) \right),
\end{align*}
where we used the reciprocity formula twice. 
\end{proof}

Putting (\ref{n3first}) and Lemma \ref{hardybothodd} together yields the first part of Theorem \ref{mainthm}.

\subsection{Proof of the Second Part of Theorem \ref{mainthm}}

Let $d, c$ be coprime integers and $d$ an even integer. As in section \ref{firstpart}, we would like to apply Mordell's Theorem \ref{mordell} to find $\# ( T(d, c) \cap (2 \Z)^2)$. However, Mordell's Theorem \ref{mordell} is no longer true if the integers $u, v, w$ are not pairwise coprime. We hence need to adjust the statement. 

\begin{prop} \label{mordell2}
Let $d, c$ be coprime integers and suppose that $d$ is even and $c$ is odd. The number of lattice points in the tetrahedron $D(d, c, 2)$ is given by \begin{align*}
    \# (D(d, c, 2) \cap \Z^3) &= \frac{cd}{3} + \frac{1}{4} (2 c + 2 d + c d) + \frac{1}{4} (c + d + 2) \\
    &+ \frac{1}{12}\left( \frac{2c}{d} + \frac{2d}{c} + \frac{cd}{2}\right) + \frac{1}{24cd} - \frac{5}{2} - (2 s(d, c) + s (2 c, d)).
\end{align*} 
\end{prop}

\begin{proof}
The proof is along the lines of Mordell's proof \cite{mordell}. We only indicate where adjustments need to be made. 

One starts with the formula $$\# (D(d, c, 2) \cap \Z^3) = \frac{1}{2} \sum_{x, y, z}{\vphantom{\sum}}' \left( \floor{\frac{x}{d} + \frac{y}{c} + \frac{z}{2}} - 1 \right)  \left( \floor{\frac{x}{d} + \frac{y}{c} + \frac{z}{2}} - 2 \right),$$
where the sum goes over $0 \leq x < d$, $0 \leq y < c$, and $0 \leq z < 2$ and $\sum'$ means that we omit $x = y = z = 0$. 

There exists exactly one lattice point $(x, y, z)$ with $\frac{x}{d} + \frac{y}{c} + \frac{z}{2} = 1$, namely $x = \frac{d}{2}$ and $z = 1$ (this would not be the case if $c$, $d$, $2$ were pairwise coprime). There exist no lattice points with $(x, y, z)$ with $\frac{x}{d} + \frac{y}{c} + \frac{z}{2} = 2$.

Now rewrite $$\# (D(d, c, 2) \cap \Z^3) =  \frac{1}{2} \sum_{x, y, z}{\vphantom{\sum}}' (E - ((E)) - 3 / 2 ) (E - ((E)) - 5 / 2) - \frac{3}{8},$$
where $E = \frac{x}{d} + \frac{y}{c} + \frac{z}{2}$ and the term $- \frac{3}{8}$ is the correction factor for the unique case where $E = 1$. 

We split up the sum $\sum_{x, y, z}' (E - ((E)) - 3 / 2) (E - ((E)) - 5 / 2) = A + B + C$ with $$A =  \frac{1}{2} \sum_{x, y, z}{\vphantom{\sum}}' (E - 3 / 2) (E - 5 / 2), \; \; B = -  \sum_{x, y, z}{\vphantom{\sum}}' (E - 2) \; ((E)), \; \; C =   \frac{1}{2} \sum_{x, y, z}{\vphantom{\sum}}' ((E))^2.$$
It is straightforward to calculate $A$, which is equal to $\frac{13 c d}{12} + \frac{c}{3d} + \frac{d}{3c} + \frac{3c}{2}+ \frac{3d}{2} - \frac{11}{4}$. For $B$, we obtain \begin{align*}
    B &= \sum_{x, y, z}{\vphantom{\sum}}' \frac{x}{d} \llrrparen{\frac{x}{d} + \frac{y}{c} + \frac{z}{2}} + \sum_{x, y, z}{\vphantom{\sum}}' \frac{y}{c} \llrrparen{\frac{x}{d} + \frac{y}{c} + \frac{z}{2}} +  \sum_{x, y, z}{\vphantom{\sum}}' \frac{z}{2} \llrrparen{\frac{x}{d} + \frac{y}{c} + \frac{z}{2}} \\
    &= \sum_{x, y}{\vphantom{\sum}}' \frac{x}{d} \llrrparen{\frac{2x}{d} + \frac{2y}{c}} + \sum_{x, y}{\vphantom{\sum}}' \frac{y}{c} \llrrparen{\frac{2x}{d} + \frac{2y}{c}} + \frac{1}{2} \sum_{y = 1}^{c - 1} \llrrparen{\frac{2y}{c}}\\
    &= \sum_{x = 1}^{d - 1} \frac{x}{d} \llrrparen{\frac{2cx}{d}} + 2 \sum_{y = 1}^{c - 1} \frac{y}{c} \llrrparen{\frac{dy}{c}} \\
    &= s(2c, d) + 2 s(d, c),
\end{align*}
where in the second last line we used that $2y$ runs through a complete set of residues modulo $c$ -- thus also $\sum_{y = 1}^{c - 1} \llrrparen{\frac{2y}{c}} = 0$ --, and that $2x$ runs through a complete set of residues of $d / 2$ twice.

To calculate $C$, we may write $$C =  \frac{1}{2} \sum_{x, y, z}{\vphantom{\sum}}' \llrrparen{\frac{x}{d} + \frac{y}{c} + \frac{z}{2}}^2 = \sum_{x, y, z}{\vphantom{\sum}}' \llrrparen{\frac{x}{2cd}}^2 - \frac{1}{4},$$
as $2cx + 2dy + cdz$ runs through a complete set of residues modulo $2cd$ with the exception that $2cx + 2dy + cdz \equiv 0 \pmod{2cd}$ is attained twice. It is then easy to calculate $C = \frac{cd}{12} + \frac{1}{24 cd}$.
\end{proof}

Similarly as in Lemma \ref{hardybothodd}, the linear combination $s(2c, d) + 2\;s(d, c)$ in Proposition \ref{mordell2} can be rewritten in terms of Hardy sums.

\begin{lem} \label{hardyoneeven}
Let $d \geq 1$ be an even integer and $c \geq 1$ such that $(d, c) = 1$. We have $$2 s(d, c) + s (2c, d) = \frac{d^2 + c^2 + 1}{6cd} - \frac{1}{2} - \frac{1}{4}S_4(c, d / 2).$$
\end{lem}

\begin{proof}
We use formula (\ref{exprhardydede}) for the Hardy sums $S_4(d, c)$. Applying the reciprocity formula for Dedekind sums, we calculate
\begin{align*}
   2 s(d, c) + s (2c, d) &= 2 \left( \frac{1}{12} \frac{d^2 + c^2 + 1}{cd} - \frac{1}{4} - s(c, d) \right) + s(c, d / 2) \\
   &= \frac{d^2 + c^2 + 1}{6cd} - \frac{1}{2} - \frac{1}{4}S_4(c, d / 2),
\end{align*}
where we used that $s(2c, 2d') = s(c, d')$ for all $d' \in \Z$. 
\end{proof}

Proposition \ref{mordell2} and Lemma \ref{hardyoneeven} prove the second part of Theorem \ref{mainthm}.

\section{Hardy Sums as Intersection Numbers and the Rademacher Cocycle} \label{geomconstsec}

As mentioned in the introduction, the Dedekind sums $s(a, c)$ in (\ref{sdefi}) are uniquely characterized by the following three properties for $(a, c) = 1$: \begin{enumerate}
    \item $s(0, 1) = 0$, 
    \item $s(a + c, c) = s(a, c)$,
    \item $s(a, c) - s(- c, a) = \frac{1}{12} \frac{a^2 + c^2 + 1}{ac} - \frac{1}{4}$ for $a, c \neq 0$.
\end{enumerate}

Since for each coprime pair of integers $a, c$ there exists a a matrix $A = \SmallMatrix{a}{*}{c}{*} \in \SL$, we may write $s(A)$ for $s(a, c)$, i.e. view the Dedekind sums as a function $s : \SL \to \Q$. By definition, the Dedekind sum $s(A)$ is invariant under right-multiplication by $T = \SmallMatrix{1}{1}{0}{1}$; moreover, we set $s(1, 0) = s(I) = 0$. The three defining properties above then correspond to the matrix operations under the generators $T$ and $S = \SmallMatrix{0}{-1}{1}{0}$ of $\SL$, namely for $A = \SmallMatrix{a}{*}{c}{*} \in \SL$ with $a, c \neq 0$: \begin{enumerate}
    \item $s(S) = s(I) = 0$, 
    \item $s(T^n A) = s(A)$ for all $n \in \Z$, 
    \item $s(A) - s(SA) = \frac{1}{12} \frac{a^2 + c^2 + 1}{cd} - \frac{1}{4}$.
\end{enumerate}
Each matrix in $\SL$ can be written as a word in $T$, $S$ and therefore we get that the three properties uniquely define the function $s : \SL \to \Q$. 

Similarly, we may view the Hardy sums $S(a, c)$ and $S_4(a, c)$ as functions on their respective groups $\G_\theta = \langle T^2, S \rangle$ and $\G^0(2) = \langle T^2, V \rangle$ with $V = \SmallMatrix{1}{0}{1}{1}$. To avoid confusion with the matrix $S$, we denote these functions in gothic print, i.e. $\Sf : \G_\theta \to \Z$ and $\Sf_4 : \G^0(2) \to \Z$. Put $\Sf(A) = S(a, c)$ for $A = \SmallMatrix{a}{*}{c}{*} \in \G_\theta$.\footnote{To make our notations simpler, we consider $S(a, c)$ instead of the usual $S(d, c)$. The change is insubstantial, as $S(d, c) = - \Sf(A^{-1}) = \Sf(A)$ for $A = \SmallMatrix{a}{b}{c}{d} \in \G_\theta$; see Proposition \ref{invprop}.} The function $\Sf : \G_\theta \to \Z$ is uniquely determined by the properties: \begin{enumerate}
    \item $\Sf(S) = \Sf(T) = \Sf(I) = 0$, 
    \item $\Sf(T^{2n}A) = \Sf(A)$ for all $n \in \Z$,
    \item $\Sf(A) - \Sf(SA) = \sgn(ac)$ for $A = \SmallMatrix{a}{*}{c}{*} \in \G_\theta$; 
\end{enumerate}

and the function $\Sf_4 : \G^0(2) \to \Z$ (similarly defined) is uniquely determined by the properties: \begin{enumerate}
    \item $\Sf_4(V) = \Sf_4(T) = \Sf_4(I) = 0$, 
    \item $\Sf_4(AT^{2n}) = \Sf_4(A)$ for all $n \in \Z$, 
    \item $\Sf_4(A) + \Sf_4(AV) = - \sgn(c(a + c))$ for $A = \SmallMatrix{a}{*}{c}{*} \in \G^0(2)$. 
\end{enumerate} 

The uniqueness of this characterization will be seen in the proof of Theorem \ref{geomchar}. Our goal in this section is to find a geometric characterization of the functions $\Sf : \G_\theta \to \Z$ and $\Sf_4 : \G^0(2) \to \Z$, which satisfies these three corresponding properties and prove a generalized reciprocity law for $\Sf$ and $\Sf_4$.

\subsection{Net of Oriented Geodesics} \label{netsec}
Consider the set of rationals $\frac{a}{c}$ with $(a, c) = 1$, $c > 1$, and $a, c$ both odd and the net of geodesics connecting two such rationals $\frac{a}{c}$ and $\frac{b}{d}$ if $ad - bc = \pm 2$. Note that these geodesics are the image of the geodesic connecting $-1$ and $1$ under $\G_\theta$. 

Now orient the geodesic connecting $\frac{a}{c}$ and $\frac{b}{d}$ such that the geodesic goes from $\frac{a}{c}$ to $\frac{b}{d}$ if $c < d$ and from $\frac{b}{d}$ to $\frac{a}{c}$ if $d < c$. In the case $c = d$, which can only occur if the geodesic connects two integers, we give it two opposing orientations. We will denote the resulting net of oriented geodesics by $\mathcal{N}_\theta$.

\begin{center}
\begin{tikzpicture}
\clip (-6, 0) rectangle (6, 6);
\draw (-6, 0) -- (6, 0);
\draw (0, 0) -- (0, 10);
%-1 to 1
\draw[decoration={markings, mark=at position 0.2 with {\arrow{>}}, mark=at position 0.3 with {\arrow{<}}},postaction={decorate}] (0, 0) circle(5.25);
\draw[decoration={markings, mark=at position 0.25 with {\arrow{>}}},postaction={decorate}] (3.5, 0) circle(1.75);
\draw[decoration={markings, mark=at position 0.25 with {\arrow{<}}},postaction={decorate}] (-3.5, 0) circle(1.75);
\draw[decoration={markings, mark=at position 0.25 with {\arrow{>}}},postaction={decorate}] (1.4, 0)  circle(0.35);
\draw[decoration={markings, mark=at position 0.25 with {\arrow{<}}},postaction={decorate}] (- 1.4, 0)  circle(0.35);
\draw[decoration={markings, mark=at position 0.25 with {\arrow{>}}},postaction={decorate}] (4.2, 0) circle(1.05);
\draw[decoration={markings, mark=at position 0.25 with {\arrow{<}}},postaction={decorate}] (- 4.2, 0) circle(1.05);
\draw[decoration={markings, mark=at position 0.25 with {\arrow{>}}},postaction={decorate}] (0.9, 0) circle(0.15);
\draw[decoration={markings, mark=at position 0.25 with {\arrow{<}}},postaction={decorate}] (-0.9, 0) circle(0.15);
\draw[decoration={markings, mark=at position 0.25 with {\arrow{<}}},postaction={decorate}] (2, 0) circle(0.25);
\draw[decoration={markings, mark=at position 0.25 with {\arrow{>}}},postaction={decorate}] (-2, 0) circle(0.25);
\draw[decoration={markings, mark=at position 0.25 with {\arrow{>}}},postaction={decorate}] (4.5, 0) circle(0.75);
\draw[decoration={markings, mark=at position 0.25 with {\arrow{<}}},postaction={decorate}] (-4.5, 0) circle(0.75);
%1 to 3
\draw[decoration={markings, mark=at position 0.43 with {\arrow{>}}, mark=at position 0.45 with {\arrow{<}}},postaction={decorate}] (10.5, 0) circle(5.25);
\draw[decoration={markings, mark=at position 0.4 with {\arrow{<}}},postaction={decorate}] (7, 0) circle(1.75);
\draw[decoration={markings, mark=at position 0.35 with {\arrow{<}}},postaction={decorate}] (9.1, 0)  circle(0.35);
\draw[decoration={markings, mark=at position 0.35 with {\arrow{<}}},postaction={decorate}] (6.3, 0) circle(1.05);
\draw[decoration={markings, mark=at position 0.35 with {\arrow{<}}},postaction={decorate}] (8.5, 0) circle(0.25);
\draw[decoration={markings, mark=at position 0.3 with {\arrow{<}}},postaction={decorate}] (6, 0) circle(0.75);
%-3 to -1
\draw[decoration={markings, mark=at position 0.07 with {\arrow{<}}, mark=at position 0.05 with {\arrow{>}}},postaction={decorate}] (-10.5, 0) circle(5.25);
\draw[decoration={markings, mark=at position 0.1 with {\arrow{>}}},postaction={decorate}] (-7, 0) circle(1.75);
\draw[decoration={markings, mark=at position 0.15 with {\arrow{>}}},postaction={decorate}] (-9.1, 0)  circle(0.35);
\draw[decoration={markings, mark=at position 0.15 with {\arrow{>}}},postaction={decorate}] (-6.3, 0) circle(1.05);
\draw[decoration={markings, mark=at position 0.15 with {\arrow{>}}},postaction={decorate}] (-8.5, 0) circle(0.25);
\draw[decoration={markings, mark=at position 0.2 with {\arrow{>}}},postaction={decorate}] (-6, 0) circle(0.75);
\end{tikzpicture}

\smallskip

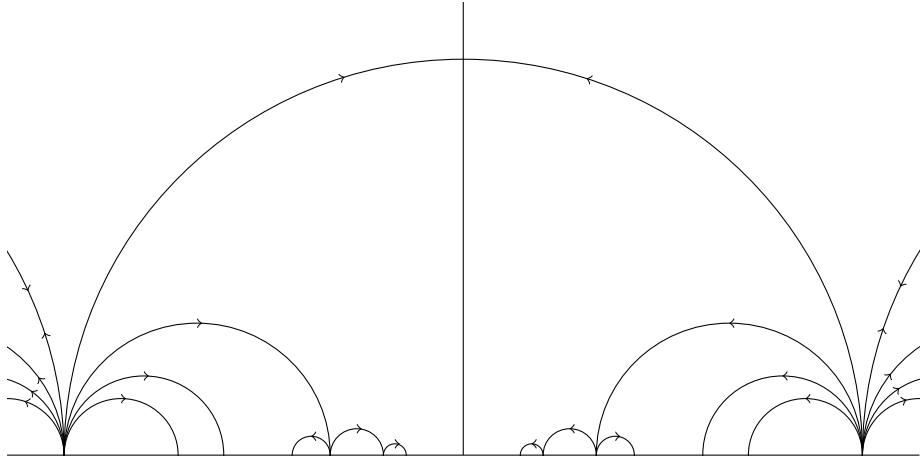
Figure 1: The net $\mathcal{N}_\theta$ for the rationals $\frac{a}{c}$, $(a, c) = 1$, $a, c$ both odd with $\vert c \vert \leq 7$.
\end{center}

\begin{lem} \label{invorient}
Let $a, b, c, d$ be odd integers with $c, d \geq 1$ and $(a, c) = 1$, $(b, d) = 1$ such that $ad - bc = \pm 2$. We have $c = d$ if and only if $c = d = 1$. We have $a = b$ if and only if $a, b = \pm 1$. If $a \neq b$, the condition $c < d$ is satisfied if and only if $\vert a \vert < \vert b \vert$. 
\end{lem}

\begin{proof}
If $a = b$, then $ad - bc = a (d - c) = \pm 2$, and since $a$ is odd, we have $a = \pm 1$. Analogously, it is showed that $c = d$ implies $c = d = 1$.

Let us thus suppose that $c, d > 1$. Then the integers $a, b$ need to have the same sign to satisfy the equation $ad - bc = \pm 2$. By multiplying the equation $ad - bc = \pm 2$ on both sides with $- 1$ if necessary, we may assume w.l.o.g. that $a, b > 0$. 

Suppose that $ad - bc = -2$ and that $a \neq b$. If $c < d$, then $ad = bc - 2$ implies that $ad < bc < bd$, hence $a < b$. If $d < c$, then $2 = bc - ad$ implies that $d(b - a) < 2$ and moreover that $b - a < 0$, since $b - a$ is an even integer and we excluded $a = b$. The case $ad - bc = 2$ is similar.
\end{proof}

Lemma \ref{invorient} has the following consequence for the oriented net of geodesics $\mathcal{N}_\theta$.

\begin{lem} \label{orientationlem}
The oriented net $\mathcal{N}_\theta$ is invariant under the maps $t_1(z) = - \overline{z}$ and $t_2(z) = z + 2$. Moreover, under $t_3(z) = \frac{z}{\vert z \vert^2}$ any oriented geodesic $g \in \mathcal{N}_\theta$ whose endpoints are not both of the form $\frac{1}{n}$ for $n \in \Z - \{0, \pm 1\}$ or both integers not equal to $\pm 1$ is mapped to an oriented geodesic in $\mathcal{N}_\theta$. 
\end{lem}

\begin{proof}
Suppose we have a geodesic from $\frac{a}{c}$ to $\frac{b}{d}$ in $\mathcal{N}_\theta$ with $c, d > 1$ odd integers, $d > c > 1$, $(a, c) = 1$ and $(b, d) = 1$ with $ad - bc = \pm 2$. We will only show the case where $a, b > 1$. 

Under the orientation-reversing map $t_1$ the geodesic is mapped to the geodesic from $- \frac{a}{c}$ to $- \frac{b}{d}$, which still lies in $\mathcal{N}_\theta$. 

Under the orientation-preserving $t_2$ the geodesic is mapped to the one from $\frac{a + 2c}{c}$ to $\frac{b + 2c}{d}$, which also still lies in $\mathcal{N}_\theta$. 

Finally, the orientation-reversing map $t_3$ sends the geodesic to the geodesic from $\frac{c}{a}$ to $\frac{d}{b}$, which as a non-oriented geodesic still lies in $\mathcal{N}_\theta$. By Lemma \ref{invorient} the orientation is preserved in $\mathcal{N}_\theta$ under $t_3$, as $a < b$.
\end{proof}

\begin{remark} \label{whathappens}
If $g \in \mathcal{N}_\theta$ in Lemma \ref{orientationlem} connects $\frac{1}{n}$ and $\frac{1}{n + 2}$ for some $n \in \Z$, then $t_3(g)$ is a geodesic connecting $n$ and $n + 2$, i.e. two integers. But geodesics connecting two integers in $\mathcal{N}_\theta$ are oriented with two opposing orientations -- by Lemma \ref{invorient}, those are, in fact, the only geodesics in $\mathcal{N}_\theta$ which have opposing orientations. Hence, the oriented geodesic $t_3(g)$ does not have two opposing orientations and does not lie in $\mathcal{N}_\theta$.
\end{remark}

\begin{remark} \label{orientationA}
More generally, as $t_2(z) = T^2.z$ and $t_1(t_3(z)) = S.z$ and $\G_\theta = \langle T^2, S \rangle$, the proof of Lemma \ref{orientationlem} shows that the map arising from each matrix $A \in \G_\theta$ preserves geodesics in $\mathcal{N}_\theta$ with the exception of the doubly-oriented geodesics connecting two integers and their preimages under $A$.
\end{remark}

\subsection{Intersection Numbers} 

Let $g, h$ be any two oriented geodesics in $\uh$. Define $\varphi_g(h)$ as follows: $$\varphi_g(h) = \begin{cases} 0, &\mbox{if } g \cap h = \emptyset, \\ + 1, &\mbox{if } g \; \mathrm{intersects \; } h \mathrm{\; on \; the \; right}, \\ - 1, &\mbox{if } g \; \mathrm{intersects \; } h \mathrm{\; on \; the \; left}. \end{cases}$$
By "$g$ intersects $h$ from the right" we mean that if we follow the path of $h$, then the orientation of $g$ at the intersection point is directed to the right. If the geodesic $g$ is doubly-oriented (as we allowed in section \ref{netsec}), then we set $\varphi_g(h) = 1 - 1 = 0$.

We will define the Hardy sums as a signed intersection number with the net $\mathcal{N}_\theta$, which we defined in section \ref{netsec}. For a geodesic $h$, define the signed intersection number \begin{equation} \label{signedintdefi}
    I_\theta(h) = \sum_{g \in \mathcal{N}_\theta} \varphi_g(h),
\end{equation}
whenever the sum exists.

We will denote an oriented geodesic $h$ from two points $z_1, z_2 \in \uh \cup \Q \cup \{i \infty \}$ by $h(z_1, z_2)$. To simplify notation, we will write $h_x = h(i \infty, x)$. 

Let $\Q_\theta = \G_\theta.i\infty$ be the set of rationals $\frac{a}{c}$ with $(a, c) = 1$ and such that $a + c$ is odd. For each $x \in \Q_\theta$, let $I_\theta(x) = I_\theta(h_x)$ be the signed intersection number of $h_x$ with the net $\mathcal{N}_\theta$. We first need to show that the function $I_\theta(x)$ is well-defined. 

\begin{lem} \label{finiteintlem}
For $x \in \Q_\theta$, the geodesic $h_x$ only intersects finitely many geodesics in $\mathcal{N}_\theta$. 
\end{lem}

\begin{proof}
Since the sign of intersection does not matter to prove the claim, we will ignore the orientation on the geodesics $\mathcal{N}_\theta$ for the remainder of this proof and refer to $\mathcal{N}_\theta$ as a set of non-oriented geodesics.

Suppose we have a geodesic in $\mathcal{N}_\theta$ connecting $0 < \frac{a}{c} < \frac{b}{d}$, which intersects $h_x$. The maximal radius of any semi-circular geodesic in $\mathcal{N}_\theta$ is $1$, since the radius is given by $$\frac{1}{2} \left( \frac{b}{d} - \frac{a}{c} \right) = \frac{1}{2} \frac{bc - ad}{cd} = \frac{1}{cd} \leq 1.$$

Consider the triangle $\Delta$ with endpoints $-1$, $1$ and $\infty$. Since $\Delta$ is the boundary of a fundamental domain for $\G_\theta$, the orbit of $\Delta$ under $\G_\theta$ gives a triangulation of the upper half plane $\uh$. Denote the collection of geodesic boundaries of the triangulation in $\uh$ by $\mathcal{N}$. In fact, the set $\mathcal{N}_\theta$ is contained in $\mathcal{N}$. We will prove the stronger claim that the geodesic $h_x$ will intersect only finitely many geodesics of $\mathcal{N}$ in $\uh$.

Any geodesic in $\mathcal{N}$ which does intersect $h_x$ is one whose endpoints are such that one is to the left of $x$ and one is to the right of $x$.  

Write $x = \frac{\alpha}{\gamma}$ with $(\alpha, \gamma) = 1$ and $\alpha + \gamma$ odd. Let $A = \SmallMatrix{\alpha}{\beta}{\gamma}{\delta} \in \G_\theta$ be the unique matrix (up to sign) such that $\Delta_x = A.\Delta$ is a triangle, whose interior is intersected by $h_x$ and has $x$ as a vertex. Explicitly, the matrix $A = \pm \SmallMatrix{\alpha}{\beta}{\gamma}{\delta} \in \G_\theta$ is determined by $\gamma > 0$ and $- \gamma < \delta < \gamma$.

\begin{center}
\begin{tikzpicture}
\clip (8, 0) rectangle (16, 5);
\draw (8, 0) -- (16, 0);
\draw [red] (12, 0) circle(2);
\draw (10.25, 0) circle(0.25); 
\draw (12.25, 0) circle(1.75);
\draw (10.5, 0) -- (10.5, 6);
\draw [red] (9.33333, 0) circle(4.6666);
\end{tikzpicture}

Figure 2: Picture of the case $x = \frac{3}{4}$. The red geodesics are those in $\mathcal{N}_\theta$, which are intersected by $h_{3 / 4}$. 
\end{center}

Let $r_x$ be the greatest radius of a geodesic bounding $\Delta_x$. Since $\G_\theta.\Delta$ is a triangulation, the radius of a geodesic in $\mathcal{N}$ whose endpoints lie to the left and right of $x$ is bounded from below by $r_x$: Any geodesic in $\mathcal{N}$, which is such that $x$ lies in between its endpoints, with smaller radius than $r_x$ would have to intersect the triangle $\Delta_x$, which cannot happen.

Thus, for any geodesic from $\frac{a}{c}$ to $\frac{b}{d}$ in $\mathcal{N}$ intersecting $h_x$ there is a lower bound on its radius $\frac{1}{cd}$ (and a fortiori in $\mathcal{N}_\theta$). More precisely, we have $$1 \leq cd \leq \frac{1}{r_x}.$$ Hence, there are only finitely many possible geodesics in $\mathcal{N}$ intersecting $h_x$. 
\end{proof}

\begin{remark} \label{remarkintersectionzero}
For instance, the geodesic $h_0$ from $i \infty$ to $0$ intersects only one geodesic in $\mathcal{N}_\theta$, the one connecting $-1$ and $1$. In the notation of the proof of Lemma \ref{finiteintlem}, the triangle $\Delta_0$ is given by $S.\Delta$. 
\end{remark}

Now we prove that the function $I_\theta(x)$ actually satisfies the three properties defining the Hardy sums $S(a, c)$. 

\begin{thm} \label{geomchar}
The signed intersection number $I_\theta : \Q_\theta \to \Z$ satisfies the following properties for $x \in \Q_\theta$: \begin{enumerate}
    \item $I_\theta(-x) = -I_\theta(x)$, 
    \item $I_\theta(x + 2) = I_\theta(x)$, 
    \item if $x \neq 0$, then $I_\theta(x) + I_\theta \left( \frac{1}{x} \right) = \sgn(x)$. 
\end{enumerate}
Moreover, the function $I_\theta$ is uniquely determined by the properties $(1)$-$(3)$.
\end{thm}

As an immediate corollary, we get that the Hardy sums $\Sf : \G_\theta \to \Z$ can be realized as a signed intersection number. 

\begin{cor}
For $A \in \G_\theta$, we have $$\Sf(A) = I_\theta (A.i \infty).$$
\end{cor}

\begin{proof}[Proof of Theorem \ref{geomchar}]
We first show that $I_\theta(x)$ in (\ref{signedintdefi}) satisfies these properties. For the first property, note that reflection along the imaginary axis sends the geodesic $h_x$ to $h_{-x}$. In particular, reflection does not change the orientation of $h_x$. However, reflection does change the orientation of the geodesics $g \in \mathcal{N}_\theta$ which intersect $h_x$. Since $\mathcal{N}_\theta$ is invariant under reflections along the imaginary axis by Lemma \ref{orientationlem}, we have $-I_\theta(x) = I_\theta(-x)$. 

For the second property, we use that $\mathcal{N}_\theta$ is invariant under translations $z \mapsto z + 2$ -- again, by Lemma \ref{orientationlem} -- and that $h_x \mapsto h_{x + 2}$ under translation. Hence $I_\theta(x + 2) = I_\theta(x)$. 

For the third property, we may assume w.l.o.g. that $x \in (-1, 1)$ by the second property. %Consider the oriented geodesic $h(0, 1 / x)$ %with $\mathcal{N}_\theta$ is equal to the number of unsigned intersections of $g_x$ with $\mathcal{N}_\theta$. Indeed, reflection at the unit circle $t_3 : z \mapsto \frac{z}{\vert z \vert^2}$ maps $g_{1 / x}'$ to $g_x$. Since reflections at the unit circle preserve $\mathcal{N}_\theta$ (as a set of non-oriented geodesics), the unsigned intersection number is the same. 
Let $n_0$ be an odd integer, such that $n_0 < \frac{1}{x} < n_0 + 2$. Let $g_0$ be the geodesic connecting $\frac{1}{n_0}$ and $\frac{1}{n_0 + 2}$. 

In $\mathcal{N}_\theta$ the geodesic $g_0$ is oriented from $\frac{1}{n_0}$ to $\frac{1}{n_0 + 2}$ if $n_0$ is positive (i.e. $\sgn(x) > 0$) and from $\frac{-1}{\vert n_0 \vert - 2}$ to $\frac{-1}{\vert n_0 \vert}$ if $n_0$ is negative (i.e. if $\sgn(x) < 0$). The geodesic $g_1$ from $n_0$ to $n_0 + 2$ has two opposing orientations in $\mathcal{N}_\theta$.

We write \begin{equation} \label{signphi}
    I_\theta(h_x) = \varphi_{g_0}(h_x) + \sum_{g \in \mathcal{N}_\theta - \{g_0\}} \varphi_g(h_x) = \sgn(x) + \sum_{g \in \mathcal{N}_\theta - \{g_0\}} \varphi_g(h_x).
\end{equation}

We have $t_3(h_x) = h(0, 1 / x)$ and by Lemma \ref{orientationlem}, we see that $$\sum_{g \in \mathcal{N}_\theta - \{g_0\}} \varphi_g(h_x) = - \sum_{g \in \mathcal{N}_\theta - \{g_1\}} \varphi_g(h(0, 1 / x)) = - I_\theta (h(0, 1 / x)),$$
since $t_3$ preserves the orientations on $\mathcal{N}_\theta$, except for geodesics connecting points $\frac{1}{n_0}$ and $\frac{1}{n_0 + 2}$ like $g_0$; see Remark \ref{whathappens}. It should also be noted that the doubly-oriented geodesic connecting $-1$ and $1$, which is intersected by both $h_x$ and $h(0, 1 / x)$ gets mapped to itself under $t_3$.

Consider the triangle of oriented geodesics from $h_0$, $h(0, 1 / x)$ and $h(1 / x, i \infty)$. It is clear that $I_\theta(h(1 / x, i \infty)) = - I_\theta(h_{1 / x})$ and that $I_\theta(h_0) + I_\theta(h(0, 1 / x)) + I_\theta(h(1 / x, i \infty)) = 0$. But by Remark \ref{remarkintersectionzero}, we have $I_\theta(h_0) = 0$, hence $I_\theta(h(0, 1 / x))= I_\theta(h_{1 / x})$, which proves the third property. 

\smallskip

Let us now show uniqueness: Let $I_1(x)$ and $I_2(x)$ be two functions satisfying the three properties. From the first property, it follows that $I_1(0) = I_2(0) = 0$. By the second property the difference $I_3(x) = I_1(x) - I_2(x)$ is zero at every integer $x \in 2\Z$. 

Suppose that $I_3(a / c)$ is zero for all $\frac{a}{c} \in \Q_\theta$, $(a, c) = 1$ with $1 \leq c < c'$. By what we just showed, this is true for $c' = 2$. We now show that the function $I_3$ also vanishes for $\frac{a'}{c'}$ with $(a', c') = 1$ and $a' + c'$ being odd, from which the claim will follow by induction. By the second property, we may assume w.l.o.g. that $-c' < a' < c'$. Using the third property, we see that $I_3(x) = - I_3\left( \frac{1}{x} \right)$. But $$I_3\left( \frac{a'}{c'} \right) = - I_3\left( \frac{c'}{a'} \right) = 0$$
by our assumption (since $\vert a' \vert < c'$).
\end{proof}

\begin{remark}
There is an analog for $\Sf_4 : \G^0(2) \to \Q$ which is obtained as an intersection number by shifting the net $\mathcal{N}_\theta$ with $z \mapsto z + 1$ and reversing the orientation of the geodesics. This follows from the fact $\Sf(x + 1) = - \Sf_4(x)$.
\end{remark}

\subsection{The Rademacher Cocycle}
We continue our investigation into the geometric nature of the function $\Sf$. We will now show that it is in fact \textit{a} splitting of the Rademacher cocycle (\ref{radecocycledefi}) restricted to the subgroup $\Gamma_\theta$. On the full modular group $\SL$, the Rademacher cocycle splits to the Dedekind symbol (\ref{phidefi}), i.e. for $A, B \in \SL$ with $A = \SmallMatrix{*}{*}{c_A}{*}$, $B = \SmallMatrix{*}{*}{c_B}{*}$, and $AB = \SmallMatrix{*}{*}{c_{AB}}{*}$, we have \begin{equation} \label{radecoceq}
    \Phi(AB) - \Phi(A) - \Phi(B) = - \sgn(c_A c_B c_{AB}).
\end{equation}

\begin{remark}
The sign function in (\ref{radecoceq}) is defined such that $\sgn(0) = 0$.
\end{remark}

\begin{thm} \label{radecocthm}
Let $A, B \in \G_\theta$ with $A = \SmallMatrix{*}{*}{c_A}{*}$, $B = \SmallMatrix{*}{*}{c_B}{*}$, and $AB = \SmallMatrix{*}{*}{c_{AB}}{*}$. We have $$\Sf(AB) - \Sf(A) - \Sf(B) = - \sgn(c_A c_B c_{AB}).$$ 
\end{thm}

%Let $\Sf : \Q_\theta \to \Z$ be the unique function from Theorem \ref{geomchar}, which is defined as a signed intersection number as in (\ref{signedintdefi}). We may view $\Sf$ as a function on the matrices $A \in \G_\theta$ by setting $\Sf(A) = \Sf(A.i\infty)$ and $\Sf(T^2) = \Sf(I) = 0$. By Remark \ref{remarkintersectionzero}, we have $\Sf(S) = 0$.

From property (3) of Theorem \ref{geomchar} (the reciprocity formula), it follows that \begin{equation} \label{reci}
    \Sf(SA) - \Sf(A)= \sgn(a c)
\end{equation}
for $A = \SmallMatrix{a}{b}{c}{d} \in \G_\theta$, which, since $\Sf(S) = 0$, proves Theorem \ref{radecocthm} in the case $B = S$.

We now proceed to prove Theorem \ref{radecocthm} through a series of proposition and lemmas.

\begin{lem} \label{reducelem}
Let $A, B \in \G_\theta$ and write $A = \SmallMatrix{a_A}{b_A}{c_A}{d_A}$, $B = \SmallMatrix{a_B}{b_B}{c_B}{d_B}$, and $AB = \SmallMatrix{*}{*}{c_{AB}}{*}$. We have $$\Sf(AB) - \Sf(A) - \Sf(B) = - \sgn(c_A c_B c_{AB})$$ if and only if $$\Sf(A'B') - \Sf(A') - \Sf(B') = - \sgn(c_{A'} c_{B'} c_{A'B'}),$$ where $A' = AS = \SmallMatrix{a_{A'}}{b_{A'}}{c_{A'}}{d_{A'}}$, $B' = SB =  \SmallMatrix{a_{B'}}{b_{B'}}{c_{B'}}{d_{B'}}$ and $A'B' = \SmallMatrix{*}{*}{c_{A'B'}}{*}$.
\end{lem}

\begin{proof}
Due to the symmetry of the claim, it suffices to prove that if $\Sf(A'B') - \Sf(A') - \Sf(B') = - \sgn(c_{A'} c_{B'} c_{A'B'})$ holds, then $\Sf(AB) - \Sf(A) - \Sf(B) = - \sgn(c_{AB} c_A c_B)$ holds as well. 

Up to replacing $A$ with $-A$ and $B$ with $-B$ we may assume w.l.o.g. that $c_A, c_B > 0$. For the matrices $A'$ and $B'$ we have $A' B' = -AB$, hence $c_{A'B'} = - c_{AB}$,  $c_{A'} = d_A$ and $c_{B'} = a_B$. By the reciprocity formula (\ref{reci}), we have $$\Sf(B') = \Sf(B) - \sgn(a_B c_B)$$
and \begin{equation*}
    \Sf(A') = \Sf(A) - \sgn(d_A c_A).
\end{equation*}

It hence follows that $$\Sf(-AB) - \Sf(A') - \Sf(B') = \Sf(AB) - \Sf(A) - \Sf(B) + \sgn(c_A d_A) + \sgn(a_B c_B).$$

If $\sgn(c_{A'} c_{B'}) < 0$, then $$\sgn(c_A d_A) + \sgn(a_B c_B) = \sgn(c_{A'}) + \sgn(c_{B'}) = 0,$$
thus $$\Sf(AB) - \Sf(A) - \Sf(B) = \Sf(A'B') - \Sf(A') - \Sf(B') = \sgn(c_{A'B'}) = - \sgn(c_{AB}).$$

If $\sgn(c_{A'} c_{B'}) > 0$, then $$\sgn(c_A d_A) + \sgn(a_B c_B) = \sgn(c_{A'}) + \sgn(c_{B'}) = 2 \; \sgn(c_{A'}) = 2 \; \sgn(c_{B'}).$$ However, as $- c_{A'B'} = c_{AB} = c_{A'}c_A + c_{B'} c_B$, we have $\sgn(c_{A'B'}) = - \sgn(c_{A'}) = - \sgn(c_{B'})$. Thus, $$\Sf(AB) - \Sf(A) - \Sf(B) = - \sgn(c_{A'B'}) - 2 \; \sgn(c_{A'}) = \sgn(c_{A'B'}) = - \sgn(c_{AB}).$$
\end{proof}

With Lemma \ref{reducelem}, we may directly prove Theorem \ref{radecocthm} for the special case $B = A^{-1}$.

\begin{prop} \label{invprop}
For all $A \in \G_\theta$, we have $\Sf(A^{-1}) = - \Sf(A)$. 
\end{prop}

\begin{proof}
Since $\Sf(A)$ is invariant under multiplication on the right by $T^2$ and under multiplication by $-I$, we may assume w.l.o.g. that $A = T^{2n_1} S T^{2n_2} S \cdots T^{2n_r} S$ for integers $n_1, ..., n_r$ and $r \geq 1$. The proof is by induction on $r$. 

Suppose $r = 1$. Then $\Sf(T^{2n_1} S) = \Sf(2n_1) = 0$ and $\Sf(S^{-1} T^{-2n_1}) = \Sf(0) = 0$ by Theorem \ref{geomchar}, so the claim is true. 

Suppose that $r > 1$ and the claim holds for $r - 1$. We have $$\Sf(A^{-1}) = \Sf(S T^{-2n_r} \cdots T^{- 2 n_2} S T^{-2n_1}) = \Sf(S T^{-2n_r} \cdots T^{- 2 n_2} S).$$ By the induction hypothesis we have $$\Sf((T^{2n_2} S \cdots T^{2n_r}S)^{-1}) = - \Sf(T^{2n_2} S \cdots T^{2n_r}S),$$
which by Lemma \ref{reducelem} implies that $$\Sf((T^{2n_2} S \cdots T^{2n_r}S)^{-1} S) = \Sf(S T^{-2n_r} \cdots T^{- 2 n_2} S) = - \Sf(S T^{2n_2} S \cdots T^{2n_r}S).$$

Since $\Sf(S T^{2n_2} S \cdots T^{2n_r}S) = \Sf(T^{2n_1} S T^{2n_2} S \cdots T^{2n_r}S)$ by Theorem \ref{geomchar}, the claim follows.
\end{proof}

\begin{lem} \label{wordlem}
Let $A = \SmallMatrix{a}{b}{c}{d} \in \G_\theta$ and write $A = \pm T^{2n_0} S \cdots T^{2n_r} S T^{2n_{r + 1}}$ for integers $n_0, ..., n_{r + 1}$ and $n_1, n_2, ..., n_r \neq 0$. We have $\frac{a}{c} \in (-1, 1)$ if and only if $n_0 = 0$.
\end{lem}

\begin{proof}
By induction on $r$. For $r = 0$, we have $\frac{a}{c} = T^{2n_0} S.i\infty = 2n_0$, for which the claim is obviously true. 

Suppose $r > 0$ and that the claim holds for $1, ..., r - 1$. Let $$A = \pm T^{2n_0} S T^{2n_1} S \cdots T^{2n_r} S T^{2n_{r + 1}} \; \; \mathrm{with} \; n_1, ..., n_r \neq 0;$$
since $S^2 = - I$, there is always such a representation. Let $\frac{a'}{c'} = T^{2n_1} S \cdots T^{2n_r} S.i\infty$, which by the induction hypothesis lies outside $(-1, 1)$. Then $$\frac{a}{c} = T^{2n_0} S T^{2n_1} \cdots T^{2n_r} S.i \infty = T^{2n_0} S.\frac{a'}{c'} = - \frac{c'}{a'} + 2n_0$$ and, as $- \frac{c'}{a'} \in (-1, 1)$, it is easy to see that $\frac{a}{c} \in (-1, 1)$ if and only if $n_0 = 0$. 
\end{proof}

The next lemma proves Theorem \ref{radecocthm} under certain extra conditions, to which the general case will be reduced.

\begin{lem} \label{helpprop}
Let $A, B \in \G_\theta$ and write $A = \SmallMatrix{a_A}{b_A}{c_A}{d_A}$, $B = \SmallMatrix{a_B}{b_B}{c_B}{d_B}$, and $AB = \SmallMatrix{a_{AB}}{b_{AB}}{c_{AB}}{d_{AB}}$. Suppose that $- \frac{d_A}{c_A} \in (-1, 1)$, $\alpha = \frac{a_A}{c_A} \in (-1, 1)$, and $\beta = \frac{a_B}{c_B} \notin (-1, 1)$. Then $$\Sf(AB) - \Sf(A) - \Sf(B) = - \sgn(c_A c_B c_{AB}).$$
 
\end{lem}

\begin{proof}
Upon replacing $A$ by $-A$ and $B$ by $-B$, we may assume w.l.o.g. that $c_A, c_B > 0$. Under these conditions $c_{AB} = c_A a_B + d_A c_B > 0$ if and only if $c_A \beta + d_A > 0$, i.e. $\beta > - \frac{d_A}{c_A}$. 

%Upon replacing $A$ by $T^{2n} A$, we may further assume that $\alpha \in (-1, 1)$. 

Note that $\frac{a_{AB}}{c_{AB}} = A.\beta = \frac{a_A}{c_A} - \frac{1}{c_A(c_A \beta + d_A)}$. Since $\vert c_A \beta + d_A \vert \geq c_A \vert \beta \vert - \vert d_A \vert > \vert c_A \vert - \vert d_A \vert \geq 1$ by assumption, it follows that $$\left\vert \frac{a_{AB}}{c_{AB}} \right\vert = \left\vert  \frac{a_A}{c_A} - \frac{1}{c_A(c_A \beta + d_A)} \right\vert < \left\vert  \frac{a_A}{c_A} \right\vert + \frac{1}{c_A} \leq 1,$$
since $\vert a_A \vert + 1 \leq \vert c_A \vert$.  

Consider the geodesic $h = h( a_{AB} / c_{AB}, \alpha)$. The geodesic $h$ is the image of $h(\beta, i \infty)$ under application of $A$. For the oriented triangle consisting of the edges $h(i \infty, a_{AB} / c_{AB})$, $h$ and $h(\alpha, i \infty)$, we have $$I_\theta(h(i \infty, a_{AB} / c_{AB})) + I_\theta(h) + I_\theta (h(\alpha, i \infty)) = 0.$$

Since $I_\theta(h(i \infty, a_{AB} / c_{AB})) = \Sf(AB)$ and $I_\theta( h(\alpha, i \infty)) = - \Sf(A)$, it suffices to prove that $I_\theta(h) = -\Sf(B) + \sgn(c_{AB})$.

By Remark \ref{orientationA}, the sign of each signed intersection of $h(\beta, i \infty)$ with geodesics in $\mathcal{N}_\theta$ does not change under the application of $A$ with the exception of the single doubly-oriented geodesics that $h(\beta, i \infty)$ intersects (as $h = A.h(\beta, i \infty)$ has both endpoints in $(-1, 1)$, the geodesic $h(\beta, i \infty)$ does not intersect any preimage of a geodesic connecting two integers under $A$). For the change of the signed intersection due to $h(\beta, i \infty)$ intersecting a doubly-oriented geodesic, we need to introduce a correction factor to $I_\theta(h(\beta, i \infty))$ to find $I_\theta (h)$. 

Thus consider the geodesic $\tilde{g}$ from $n_0$ to $n_0 + 2$ with $n_0$ an odd integer such that $n_0 < \beta < n_0 + 2$, which is the only doubly-oriented geodesic in $\mathcal{N}_\theta$ that $h(\beta, i \infty)$ intersects. 

Under the matrix $A$, the geodesic $\tilde{g}$ is mapped to the geodesic in $A.\tilde{g}$ connecting $\frac{a_A n_0 + b_A}{c_A n_0 + d_A}$ and $\frac{a_A(n_0 + 2) + b_A}{c_A(n_0 + 2) + d_A}$. To find the orientation of the geodesic $A.\tilde{g}$ in $\mathcal{N}_\theta$, we need to find out whether $\vert c_A n_0 + d_A \vert < \vert c_A n_0 + d_A + 2c_A \vert$ or the contrary holds. Since $c_A > 0$, this only depends on whether $c_A n_0 + d_A$ is negative or not. 

Indeed, the linear function $\R \to \R$, $x \mapsto c_A x + d_A$ switches from negative to positive values at $x = - \frac{d_A}{c_A}$. But this can't happen in the interval $(n_0, n_0 + 2)$, as $- \frac{d_A}{c_A} \in (-1, 1)$ and $\beta \in (n_0, n_0 + 2)$ does not lie in $(-1, 1)$ by assumption. 

Thus, the orientation of $A.\tilde{g}$ in $\mathcal{N}_\theta$ only depends on the sign of $c_A \beta + d_A = c_{AB}$. Namely, as $A.(n_0 + 2) - A.n_0 = \frac{1}{(c_A (n_0 + 2) + d_A) (c_A n_0 + d_A)} > 0$, the orientation of $A.\tilde{g}$ is positive, if $c_{AB}$ is negative, and vice versa. Note further that $A.\infty - A.x = \frac{1}{c_A(c_A x + d_A)}$ for any $x \in \R$. Since $c_An_0 + d_A$ and $c_A (n_0 + 2) + d_A$ have the same sign, the numbers $\alpha - A.n_0$ and $\alpha - A.(n_0 + 2)$ have the same sign as well. In other words, $\alpha$ lies outside the interval bounded by $A.n_0$ and $A.(n_0 + 2)$ or the geodesic $h$ from $\frac{a_{AB}}{c_{AB}}$ to $\alpha$ is oriented outwards of the geodesic $A.\tilde{g}$. 

Now if the orientation is positive, then the geodesic $h$ intersects $A.\tilde{g}$ with $A.\tilde{g}$ pointing to the left and if the orientation is negative, then $h$ intersects $A.\tilde{g}$ with $A.\tilde{g}$ pointing to the right, i.e. $$I_\theta(h) = \sgn(c_{AB}) + I_\theta(h(\beta, i \infty)) = \sgn(c_{AB}) - \Sf(B).$$
\end{proof}

All cases that we need to consider to prove Theorem \ref{radecocthm} are covered in the lemmas and propositions above. We are only left with reducing the general case to them.

\begin{proof}[Proof of Theorem \ref{radecocthm}]
Suppose $A \neq B^{-1}$ otherwise the claim follows by Proposition \ref{invprop}. 

Write \begin{equation} \label{word}
    A = \pm T^{2n_0} S T^{2n_1} S \cdots T^{2n_r} S T^{2n_{r + 1}} \; \mathrm{and} \; B = \pm T^{2m_0} S T^{2m_1} S \cdots T^{2m_t} S T^{2m_{t + 1}}
\end{equation}
for $r, t \geq 0$ and $n_1, ..., n_r, m_1, ..., m_t \neq 0$.
We may assume that $- \frac{d_A}{c_A} \in (-1, 1)$ and $\frac{a_A}{c_A} \in (-1, 1)$ by Lemma \ref{wordlem} (i.e. that $n_0 = n_{r + 1} = 0$) upon replacing $A$ with $T^{-2n_0} A T^{-2n_{r + 1}}$ and $B$ with $T^{2n_{r + 1}} B$. Again, we may assume w.l.o.g. that $c_A, c_B > 0$. 

If $\vert B.i\infty \vert > 1$ (i.e. $m_0 \neq 0$ in (\ref{word}) by Lemma \ref{wordlem}), the claim follows from Lemma \ref{helpprop}. Suppose thus that $\vert B.i \infty \vert < 1$, i.e. that $B = \pm S T^{2m_1} S \cdots T^{2m_t} S T^{2m_{t + 1}}$. We may now replace $A$ by $A' = AS$ and $B$ by $B' = SB$ as in Lemma \ref{reducelem}. 

Again, by replacing $A' = \pm ST^{2n_1} S \cdots S T^{2n_{r}}$ by $A' T^{-2n_r}$ and $B'$ by $$T^{2n_r}B' = \pm T^{2(m_1 + n_r)} S \cdots T^{2m_t} S T^{2m_{t + 1}},$$ we may assume that $- \frac{d_{A'}}{c_{A'}} \in (-1, 1)$ and that $c_{A'}, c_{B'} > 0$. If $\vert B'.i \infty \vert > 1$, i.e. if $m_1 \neq - n_r$, we may use Lemma \ref{helpprop}. Otherwise, repeat the argument until we reach $\vert B'.i\infty \vert > 1$. 

This process will end, for if $r \neq t$, one of the matrices will at some point be equal to the identity, for which the claim is trivial. If $r = t$, we must have $-m_i \neq n_{t + 1 - i}$ for an $i = 0, ..., t + 1$, since we excluded the case $A = B^{-1}$.
\end{proof}

\begin{cor}
Let $\Phi : \G_\theta \to \Z$ be the classical Dedekind symbol restricted to the subgroup $\G_\theta$. For the map $\Sf : \G_\theta \to \Z$ we have $$\Sf = \Phi - \chi_\theta,$$
where $\chi_\theta$ is the unique homomorphism $\G_\theta \to \Z$ given by $\chi_\theta(T^2) = 2$ and $\chi_\theta(S) = 0$. 
\end{cor}

\begin{proof}
We have $\Phi(AB) - \Phi(A) - \Phi(B) = - \sgn(c_A c_B c_{AB})$ by (\ref{radecoceq}). Hence, $\chi_\theta = \Phi - \Sf$ is a group homomorphism by Theorem \ref{radecocthm}. The values of $\chi_\theta$ at the generators of $\G_\theta$ are obtained by $\Phi(T^2) = 2$, $\Phi(S) = 0$ and $\Sf(T^2) = 0$, $\Sf(S) = 0$. 
\end{proof}

%Note that we recovered the formula $$\log \theta(A.z) - \log \theta(z) = \frac{\pi i}{12} \int_z^{A.z} (E_2(w) - R(w)) dw$$
%in the theory of automorphic forms. 

\begin{remark}
The same arguments work for the function $\Sf_4$ as well. In particular, the function $\Sf_4$ satisfies the generalized reciprocity law in Theorem \ref{radecocthm}.
\end{remark}

\subsection*{Acknowledgments} We would like to thank \"Ozlem Imamo\={g}lu and Tim Gehrunger for helpful comments on the presentation of this note. The author was supported by the SNF project 200021\_185014.

\end{document}